\documentclass[12pt]{article}
\usepackage{amssymb}	
\usepackage{amsmath}    
\usepackage{epsfig}
\usepackage{latexsym}
\usepackage{multirow}
\usepackage{longtable}

\newenvironment{proof}[1][Proof:]{\begin{trivlist} 
\item[\hskip \labelsep {\bfseries #1}]}{\end{trivlist}} 
\newcommand{\qed}{\nobreak \ifvmode \relax \else \ifdim\lastskip<1.5em \hskip-\lastskip \hskip1.5em plus0em minus0.5em \fi \nobreak \vrule height0.75em width0.5em depth0.25em\fi}

\def\R{{\bf R}}

\def\T{{\rm T}}

\newtheorem{algorithm}{Algorithm}[section]
\newtheorem{example}{Example}[section]
\newtheorem{theorem}{Theorem}[section]
\newtheorem{lemma}{Lemma}[section]

\newtheorem{remark}{Remark}[section]
\newcommand{\diag}{\mathop{\mathrm{diag}}}

\begin{document}
\title{Arc-Search Infeasible Interior-Point Algorithm for Linear Programming}
\author{Yaguang Yang\footnote{\normalsize NRC, Office of Research, 21 Church Street, 
Rockville, 20850. Email: yaguang.yang@verizon.net.} \\
}
\date{\today}

\maketitle    

\begin{abstract}
Mehrotra's algorithm has been the most successful infeasible interior-point algorithm for linear 
programming since 1990. Most popular interior-point software packages for linear programming are based
on Mehrotra's algorithm. This paper proposes an alternative algorithm, arc-search
infeasible interior-point algorithm. We will demonstrate, by testing Netlib problems and comparing
the test results obtained by arc-search infeasible interior-point algorithm and Mehrotra's algorithm,
that the proposed arc-search infeasible interior-point algorithm is a more efficient algorithm 
than Mehrotra's algorithm.

\end{abstract}


\newpage
 
\section{Introduction}

Interior-point method is now regarded as a mature technique of linear programming \cite[page 2]{wright97},
following many important developments in 1980-1990, such as, a proposal of path-following method
\cite{megiddo89}, the establishment of polynomial bounds for 
path-following algorithms \cite{kmy89a,kmy89b}, the development of Mehrotra's predictor-corrector (MPC) 
algorithm \cite{Mehrotra92} and independent implementation and verification \cite{lms91,lms92},
and the proof of the polynomiality of infeasible interior-point algorithm \cite{mizuno94,zhang94}. 
Although many more algorithms have been proposed since then (see, for example,
\cite{mty93,miao96,roos06,spt07,kah13}), there is no significant improvement 
in the best polynomial bound for interior-point algorithms, and there is no 
report of a better algorithm than MPC for general linear programming problems\footnote{There are 
some noticeable progress focused on problems with special structures, for example \cite{wta14}.}. 
In fact, the most popular interior-point method software packages
implemented MPC, for example, LOQO \cite{loqo}, PCx \cite{cmww97} and LIPSOL \cite{zhang96}.

However, there were some interesting results obtained in recent years. For example, higher-order 
algorithms that used second or higher-order derivatives were demonstrated to improve the computational 
efficiency \cite{Mehrotra92,lms92}. Higher-order algorithms, however, had either a poorer polynomial bound 
than first-order algorithms \cite{mar90} or did not even have a polynomial bound \cite{Mehrotra92,cartis09}. 
This dilemma was partially solved in \cite{yang13} which proved that higher-order algorithms can achieve the
best polynomial bound. An arc-search interior-point algorithm for linear programming was devised
in \cite{yang13}. The algorithm utilized the first and second-order 
derivatives to construct an ellipse to approximate the central path. Intuitively, searching along this 
ellipse should generate a larger step size than searching along any straight line. Indeed, it was shown in 
\cite{yang13} that the arc-search algorithm has the best polynomial bound and it may be very efficient 
in practical computation. This result was extended to prove a similar result for convex quadratic programming
and the numerical test result was very promising \cite{yang11}. 

The algorithms proposed in \cite{yang13,yang11} assume that the starting point is feasible and the 
central path does exist. Available Netlib test problems are limited because most Netlib problems may 
not even have an interior-point as noted in \cite{cg06}. To better demonstrate the claims in the previous 
papers, we propose an infeasible arc-search interior-point algorithm in this paper, which allows us to 
test a lot more Netlib problems. The proposed algorithm keeps a nice feature developed in \cite{yang13,yang11}, 
i.e., it searches optimizer along an arc (part of an ellipse). It also adopts some strategies used 
in MPC, such as using different step sizes for the vector of primal variables and the vector of slack 
variables. We will show that the proposed arc-search infeasible interior-point algorithm is very competitive in 
computation by testing all Netlib problems in standard form and comparing the results to those obtained 
by MPC. To have a fair comparison, both algorithms are implemented in MATLAB; for all test problems, 
the two Matlab codes use the same pre-processor, start from the same initial point, use the same parameters, 
and terminate with the same stopping criterion. Since the main cost in computation for both algorithms 
is to solve linear systems of equations which are exactly the same for 
both algorithms, and the arc-search infeasible interior-point algorithm uses 
less iterations in most tested problems than MPC, we believe that the proposed
algorithm is more attractive than the MPC algorithm.

The remaining of the paper is organized as follows. Section 2 briefly 
describes the problem. Section 3 presents the proposed algorithm and some simple but important properties. 
Section 4 discusses implementation details for both algorithms. Section 5 provides numerical results and compares the results obtained by both arc-search method
and Mehrotra's method. Conclusions are summarized in Section 6.

\section{Problem Descriptions}

Consider the Linear Programming in the standard form:
\begin{eqnarray}
\min \hspace{0.05in} c^{\T}x, \hspace{0.15in} \mbox{\rm subject to} 
\hspace{0.1in}  Ax=b, \hspace{0.1in} x \ge 0,
\label{LP}
\end{eqnarray}
where $A \in {\bf R}^{m \times n}$, $b \in {\bf R}^{m} $, $c \in {\bf R}^{n}$ 
are given, and $x \in {\bf R}^n$  is the vector to be optimized. Associated 
with the linear programming is the dual programming that is also presented in the standard form:
\begin{eqnarray}
\max \hspace{0.05in} b^{\T}\lambda, \hspace{0.15in} \mbox{\rm subject to} 
\hspace{0.1in}  A^{\T}\lambda+s=c, \hspace{0.1in} s \ge 0,
\label{DP}
\end{eqnarray}
where dual variable vector $\lambda \in {\bf R}^{m}$, and dual slack vector 
$s \in {\bf R}^{n}$. 

Throughout the paper, we will denote the residuals of the equality constraints 
(the deviation from the feasibility) by
\begin{equation}
r_b=Ax-b, \hspace{0.1in} r_c=A^{\T}\lambda + s-c,
\label{residuals}
\end{equation}
the duality measure by 
\begin{equation}
\mu=\frac{x^{\T}s}{n},
\label{duality}
\end{equation}
the $i$th component of $x$ by $x_i$, the Euclidean norm of $x$ by $\| x \|$, the 
identity matrix of any dimension by $I$, the vector of all ones with appropriate 
dimension by $e$, the Hadamard (element-wise) product of two vectors $x$ and $s$ 
by $x \circ s$. To make the notation simple for block column vectors, we will denote, 
for example, a point in the primal-dual problem $[x^{\T}, \lambda^{\T}, s^{\T}]^{\T}$ 
by $(x,\lambda, s)$. We will denote a {\it vector} initial point of any algorithm by $(x^0,\lambda^0, s^0)$,
the corresponding {\it scalar} duality measure by $\mu_0$, the point after the $k$th iteration by 
$(x^k,\lambda^k, s^k)$, the corresponding duality measure by $\mu_k$, the optimizer by 
$(x^*, \lambda^*, s^*)$, the corresponding duality measure by $\mu_*$. 
For $x \in \R^n$, we will denote a related diagonal matrix 
by $X \in \R^{n \times n}$ whose diagonal elements are components
of the vector $x$. 

The central path ${\cal C}(t)$ of the primal-dual linear programming problem
is parameterized by a scalar $t \ge 0$ as follows. For each interior 
point $(x, \lambda, s) \in {\cal C}(t)$ on the central path, there 
is a $t \ge 0$ such that 
\begin{subequations}
\begin{align}
Ax=b  \label{patha} \\
A^{\T}\lambda+s=c \label{pathb} \\
(x,s) \ge 0 \label{pathd}  \\
x_is_i = t, \hspace{0.1in} i=1,\ldots,n \label{pathc}.
\end{align}
\label{centralpath}
\end{subequations}
\noindent
Therefore, the central path is an arc in ${\bf R}^{2n+m}$ parameterized as a 
function of $t$ and is denoted as 
\begin{equation}
{\cal C}(t) = \lbrace(x(t), \lambda(t), s(t)): t \ge 0 \rbrace.
\end{equation} 
As $t \rightarrow 0$, the central path $(x(t), \lambda(t), 
s(t))$ represented by (\ref{centralpath}) approaches to a solution of LP 
represented by (\ref{LP}) because (\ref{centralpath}) reduces to the KKT condition
as $t \rightarrow 0$. 

Because of high cost of finding an initial feasible point and the central path described in (\ref{centralpath}), 
we consider  a modified problem which allows infeasible initial point.
\begin{subequations}
\begin{align}
Ax-b = r_b \label{ipatha} \\
A^{\T}\lambda+s-c = r_c \label{ipathb} \\
(x,s) \ge 0 \label{ipathd}  \\
x_is_i = t, \hspace{0.1in} i=1,\ldots,n \label{ipathc}.
\end{align}
\label{modifiedcentralpath}
\end{subequations}
We search the optimizer along an infeasible central path neighborhood. 
The infeasible central path neighborhood ${\cal F}(\gamma)$ considered in this paper is defined as a 
collection of points that satisfy the following conditions,
\begin{equation}
{\cal F}(\gamma(t))=\lbrace(x, \lambda, s): \|(r_b(t), r_c(t)) \| \le 
\gamma(t)  \| (r_b^0, r_c^0)  \|,\hspace{0.01in} 
(x,s) > 0  \rbrace,
\label{infeasible}
\end{equation}
where $r_b(1)=r_b^0$, $r_c(1)=r_c^0$, $\gamma(t) \in [0, 1]$ is a monotonic function of $t$ such that 
$\gamma(1)=1$ and $\gamma(t) \rightarrow 0$ as $t \rightarrow 0$. It is worthwhile to note that this 
central path neighborhood is the widest in any neighborhood considered in existing literatures.

\section{Arc-Search Algorithm for Linear Programming}

Starting from any point $(x^0, \lambda^0, s^0)$ in a central path neighborhood that satisfies $(x^0,s^0)>0$, 
for $k \ge 0$, we consider a special arc parameterized by $t$ and defined by the current iterate as follows:
\begin{subequations}
\begin{align}
Ax(t)-b =t r_b^k, \label{arca} \\
A^{\T}\lambda(t)+s(t)-c = t r_c^k, \label{arcb} \\
(x(t),s(t)) > 0, \label{arcd}  \\
x(t) \circ s(t) = t x^k \circ s^k.  \label{arcc}
\end{align}
\label{neiborArc}
\end{subequations}
Clearly, each iteration starts at $t=1$; and $(x(1), \lambda (1), s(1))= (x^k, \lambda^k, s^k)$. We
want the iterate stays inside ${\cal F}(\gamma)$ as $t$ decreases. We denote the infeasible central path 
defined by (\ref{neiborArc}) as 
\begin{equation}
{\cal H}(t)= \lbrace(x(t), \lambda(t), s(t)): t \ge \tau \ge 0 \rbrace.
\end{equation}
If this arc is inside ${\cal F}(\gamma)$ for $\tau =0$, then as $t \rightarrow 0$, 
$(r_b (t), r_c (t)):=t (r_b^k, r_c^k)  \rightarrow 0$; 
and equation (\ref{arcc}) implies that $\mu (t) \rightarrow 0$;
hence, the arc will approach to an optimal solution of (\ref{LP}) because (\ref{neiborArc}) 
reduces to KKT condition as $t \rightarrow 0$. To avoid computing the entire infeasible central path ${\cal H}(t)$, 
we will 
search along an approximation of ${\cal H}(t)$ and keep the iterate stay in ${\cal F}(\gamma)$.
Therefore, we will use an ellipse ${\cal E}(\alpha)$ \cite{carmo76} in $2n+m$ dimensional 
space to approximate the infeasible central path ${\cal H}(t)$, where ${\cal E}(\alpha)$ is given by
\begin{equation}
{\cal E}(\alpha)=\lbrace (x(\alpha), \lambda(\alpha), s(\alpha)): 
(x(\alpha), \lambda(\alpha), s(\alpha))=
\vec{a}\cos(\alpha)+\vec{b}\sin(\alpha)+\vec{c} \rbrace,
\label{ellipse}
\end{equation}
$\vec{a} \in \R^{2n+m}$ and $\vec{b} \in \R^{2n+m}$ are the axes of the 
ellipse, and $\vec{c} \in \R^{2n+m}$ is the center of the ellipse. Given the current iterate
$y=(x^k, \lambda^k, s^k)=(x(\alpha_0), \lambda(\alpha_0), s(\alpha_0)) \in {\cal E}(\alpha)$ 
which is also on ${\cal H}(t)$, we will determine 
$\vec{a}$, $\vec{b}$, $\vec{c}$ and $\alpha_0$ such that 
the first and second derivatives of ${\cal E}(\alpha)$ at 
$(x(\alpha_0), \lambda(\alpha_0), s(\alpha_0))$ are the same as those of
${\cal H}(t)$ at $(x(\alpha_0), \lambda(\alpha_0), s(\alpha_0))$. Therefore, by taking the first 
derivative for (\ref{neiborArc}) at 
$(x(\alpha_0), \lambda(\alpha_0), s(\alpha_0)) = (x^k, \lambda^k, s^k) \in \cal{E}$, we have
\begin{equation}
\left[
\begin{array}{ccc}
A & 0 & 0\\
0 & A^{\T} & I \\
S^k & 0 & X^k
\end{array}
\right]
\left[
\begin{array}{c}
\dot{{x}} \\ \dot{\lambda}  \\  \dot{{s}}
\end{array}
\right]
=\left[
\begin{array}{c}
r_b^k \\ r_c^k \\ {x^k} \circ {s^k} 
\end{array}
\right],
\label{doty}
\end{equation}

These linear systems of equations are very similar to those used in \cite{yang13} except that equality 
constraints in (\ref{centralpath}) are not assumed to be satisfied. By taking the second derivative, we have 
\begin{equation}
\left[
\begin{array}{ccc}
A & 0 & 0\\
0 & A^{\T} & I \\
S^k & 0 & X^k
\end{array}
\right]
\left[
\begin{array}{c}
\ddot{x} \\ \ddot{\lambda} \\  \ddot{s} 
\end{array}
\right]
=\left[
\begin{array}{c}
0 \\ 0 \\ -2\dot{x} \circ \dot{s}
\end{array}
\right].
\label{tmp}
\end{equation}
Similar to \cite{Mehrotra92}, we modify (\ref{tmp}) slightly to make sure that a substantial segment
of the ellipse stays in ${\cal F}(t)$, thereby making sure that the step size along the ellipse is significantly 
greater than zero,
\begin{equation}
\left[
\begin{array}{ccc}
A & 0 & 0\\
0 & A^{\T} & I \\
S^k & 0 & X^k
\end{array}
\right]
\left[
\begin{array}{c}
\ddot{x}(\sigma_k) \\ \ddot{\lambda}(\sigma_k)  \\  \ddot{s}(\sigma_k)
\end{array}
\right]
=\left[
\begin{array}{c}
0 \\ 0 \\ \sigma_k \mu_k e -2\dot{x} \circ \dot{s}
\end{array}
\right],
\label{ddoty}
\end{equation}
where the duality measure $\mu_k$ is evaluated at $(x^k, \lambda^k, s^k)$,
and we set the centering parameter $\sigma_k$ satisfying $0< \sigma_k < \sigma_{\max} \le 0.5$. 
We emphasize that the second derivatives are functions of $\sigma_k$ which is selected by
using a heuristic of \cite{Mehrotra92} to speed up the convergence of the proposed algorithm.
Several relations follow immediately from (\ref{doty}) and (\ref{ddoty}).
\begin{lemma}
Let $(\dot{x}, \dot{\lambda}, \dot{s})$ and $(\ddot{x}, \ddot{\lambda}, \ddot{s})$ be defined in (\ref{doty}) 
and (\ref{ddoty}). Then, the following relations hold.
\begin{equation}
s^{\T} \dot{x} + x^{\T} \dot{s} = x^{\T} {s}=n\mu,
\hspace{0.1in}
s^{\T} \ddot{x} + x^{\T} \ddot{s} = \sigma \mu n-2\dot{x}^{\T} \dot{s},
\hspace{0.1in}
\ddot{x}^{\T} \ddot{s} =0.
\end{equation}
\label{simple}
\end{lemma}

Equations (\ref{doty}) and (\ref{ddoty}) can be solved in either unreduced form, or augmented system form, 
or normal equation form as suggested in \cite{wright97}. We solve the normal equations 
for $(\dot{x}, \dot{\lambda}, \dot{s})$ and $(\ddot{x}, \ddot{\lambda}, \ddot{s})$ as follows:
\begin{subequations}
\begin{gather}
( {A}XS^{-1} {A}^{\T}) \dot{\lambda}=A X S^{-1} r_c -b, \label{pl2} \\
\dot{s}= r_c-{A}^{\T} \dot{\lambda}, \label{ps2}  \\
\dot{x}= x-XS^{-1} \dot{s}, \label{px2}
\end{gather}
\label{doy2}
\end{subequations}
and
\begin{subequations}
\begin{gather}
( {A}XS^{-1} {A}^{\T}) \ddot{\lambda}=-{A}S^{-1} (\sigma \mu n-2\dot{x} \circ \dot{s}), 
\label{ddl2} \\
\ddot{s}=-{A}^{\T} \ddot{\lambda}. 
\label{dds2} \\
\ddot{x}=S^{-1}(\sigma \mu e-X \ddot{s}-2\dot{x} \circ \dot{s}). 
\label{ddx2}
\end{gather}
\label{ddoy2}
\end{subequations}
Given the first and second derivatives defined by (\ref{doty}) and (\ref{ddoty}),
an analytic expression of the ellipse that is used to approximate the infeasible central path is derived in \cite{yang13}.  
\begin{theorem}
Let $(x(\alpha),\lambda(\alpha),s(\alpha))$ be an arc defined by 
(\ref{ellipse}) passing
through a point $(x,\lambda,s)\in {\cal E} \cap {\cal H}$, and its first and second 
derivatives at $(x,\lambda,s)$ be $(\dot{x}, \dot{\lambda}, \dot{s})$ and 
$(\ddot{x}, \ddot{\lambda}, \ddot{s})$ which are defined by
(\ref{doty}) and (\ref{ddoty}). Then the ellipse approximation of ${\cal H}(t)$ is given by
\begin{equation}
x(\alpha,\sigma) = x - \dot{x}\sin(\alpha)+\ddot{x}(\sigma) (1-\cos(\alpha)).
\end{equation}
\begin{equation}
\lambda(\alpha,\sigma) = \lambda-\dot{\lambda}\sin(\alpha)
+\ddot{\lambda}(\sigma) (1-\cos(\alpha)).
\end{equation}
\begin{equation}
s(\alpha,\sigma) = s - \dot{s}\sin(\alpha)+\ddot{s}(\sigma) (1-\cos(\alpha)).
\end{equation}
\label{ellipseSX}
\end{theorem}
In the algorithm proposed below, we suggest taking step size $\alpha^s_k =\alpha^{\lambda}_k$ which may not be 
equal to the step size of $\alpha^x_k$. 

\begin{algorithm} {\ } \\ 
Data: $A$, $b$, $c$, and step scaling factor $\beta \in (0,1)$.  {\ } \\
Initial point: ${\lambda}^0=0$, $x^0>0$, $s^0>0$, and 
$\mu_0= \frac{{x^0}^{\T}s^0}{n}$.    {\ } \\
{\bf for} iteration $k=0,1,2,\ldots$
\begin{itemize}
\item[] Step 1: Calculate $(\dot{x},\dot{\lambda},\dot{s})$ using (\ref{doy2}) and set
\begin{subequations}
\begin{gather}
\alpha_x^a:=\arg \max \{ \alpha \in [0,1] | x-\alpha \dot{x} \ge 0 \},
\\
\alpha_s^a:=\arg \max \{ \alpha \in [0,1] | s-\alpha \dot{s} \ge 0 \}.
\end{gather}
\end{subequations} 
\item[] Step 2: Calculate $\mu^a =\frac{(x+\alpha^a_x)^{\T}( s+\alpha^a_s)}{n}$ and compute the centering parameter
\begin{equation}
\sigma=\left( \frac{\mu^a}{\mu} \right)^3.
\end{equation}
\item[] Step 3: Computer $(\ddot{x},\ddot{\lambda},\ddot{s})$ using (\ref{ddoy2}).
\item[] Step 4: Set
\begin{subequations}
\begin{align}
\alpha^x = \arg \max\{ \alpha \in [0,\frac{\pi}{2}] | x^k - \dot{x}\sin(\alpha)+\ddot{x} (1-\cos(\alpha)) \ge 0  \}, 
\label{updatex} \\
\alpha^s = \arg \max\{ \alpha \in [0, \frac{\pi}{2}] | s^k - \dot{s}\sin(\alpha)+\ddot{s} (1-\cos(\alpha)) \ge 0 \}. 
\label{updates} 
\end{align}
\label{update1}
\end{subequations}
\item[] Step 5: Scale the step size by $\alpha^x_k=\beta \alpha^x$ and
$\alpha^s_k=\beta \alpha^s$ such that the update
\begin{subequations}
\begin{gather}
x^{k+1}= x^k - \dot{x}\sin(\alpha^x_k)+\ddot{x} (1-\cos(\alpha^x_k))>0, \\
\lambda^{k+1} =\lambda^k -\dot{\lambda }\sin(\alpha^s_k)
+\ddot{\lambda} (1-\cos(\alpha^s_k)),  
\\
s^{k+1} = s^k - \dot{s}\sin(\alpha^s_k)+\ddot{s} (1-\cos(\alpha^s_k))>0. 
\end{gather}
\label{update2}
\end{subequations}
\item[] Step 6: Set $k \leftarrow k+1 $. Go back to Step 1.
\end{itemize}
{\bf end (for)} 
\hfill \qed
\label{mainAlgo3}
\end{algorithm}

\begin{remark}
The main difference between the proposed algorithm and Mehrotra's algorithm is in Steps 4 and 5 where 
the iterate moves along the ellipse instead of a straight line. More specifically, instead of using 
(\ref{update1}) and (\ref{update2}), Mehrotra's method uses
\begin{subequations}
\begin{align}
\alpha^x = \arg \max\{ \alpha \in [0,1] | x^k - \alpha (\dot{x}- \ddot{x}) \ge 0  \}, \\
\alpha^s = \arg \max\{ \alpha \in [0,1] | s^k - \alpha(\dot{s}- \ddot{s} )\ge 0 \}. 
\end{align}
\label{update1a}
\end{subequations}
and
\begin{subequations}
\begin{gather}
x^{k+1}= x^k - \alpha^x_k (\dot{x} - \ddot{x}) >0, \\
\lambda^{k+1} =\lambda^k -\alpha^s_k (\dot{\lambda }
- \ddot{\lambda} ),  
\\
s^{k+1} = s^k - \alpha^s_k (\dot{s}- \ddot{s})>0. 
\end{gather}
\label{update2a}
\end{subequations}
Note that the end points of arc-search algorithm $(\alpha^x_k, \alpha^s_k)=(0,0)$
or $(\alpha^x_k, \alpha^s_k)=(\frac{\pi}{2},\frac{\pi}{2})$ in (\ref{update1}) and 
(\ref{update2}) are equat to the end points of Mehrotra's formulae in (\ref{update1a})
and (\ref{update2a}); for any 
$(\alpha^x_k, \alpha^s_k)$ between $(0, \frac{\pi}{2})$, the ellipse is a
better approximation of the infeasible central path. Therefore, the proposed algorithm should
have a larger step size than Mehrotra's method and be more efficient. 
This intuitive has been verified in our numerical test.
\end{remark}

The following lemma shows that searching along the ellipse in iterations will reduce the residuals of the 
equality constraints to zero as $k \rightarrow \infty$ provided that $\alpha^x_k$ and $\alpha^s_k$ are 
bounded below from zero.

\begin{lemma}
Let $r_b^k = Ax^k-b$, $r_c^k=A^{\T} \lambda^k +s^k -c$, 
$\varrho_k=\prod_{j=0}^{k-1} (1-\sin(\alpha^x_j))$. and 
$\nu_k=\prod_{j=0}^{k-1} (1-\sin(\alpha^s_j))$. Then, the following 
relations hold.
\begin{subequations}
\begin{align}
r_b^k = r_b^{k-1} (1-\sin(\alpha^x_{k-1})) = \cdots = r_b^0 \prod_{j=0}^{k-1} (1-\sin(\alpha^x_j))
= r_b^0 \varrho_k, \label{a}
\\
r_c^k = r_c^{k-1} (1-\sin(\alpha^s_{k-1})) = \cdots = r_c^0 \prod_{j=0}^{k-1} (1-\sin(\alpha^s_j))
= r_c^0 \nu_k. \label{b}
\end{align}
\label{errorUpdate}
\end{subequations}
\label{basic}
\end{lemma}
\begin{proof}
From Theorem \ref{ellipseSX}, searching along ellipse generates iterate
as follows.
\begin{subequations}
\begin{gather}
x^{k+1} - x^k= - \dot{x}\sin(\alpha^x_k)+\ddot{x} (1-\cos(\alpha^x_k)), \nonumber \\
\lambda^{k+1} - \lambda^k=-\dot{\lambda }\sin(\alpha^s_k)
+\ddot{\lambda} (1-\cos(\alpha^s_k)),  \nonumber \\
s^{k+1} - s^k = - \dot{s}\sin(\alpha^s_k)+\ddot{s} (1-\cos(\alpha^s_k)). \nonumber
\end{gather}
\label{diff}
\end{subequations}
In view of (\ref{doty}) and (\ref{ddoty}), we have 
\begin{eqnarray}
r_b^{k+1}-r_b^{k} & = & A (x^{k+1}-x^{k}) = A(-\dot{x} \sin(\alpha^x_{k})+\ddot{x}(1-\cos(\alpha^x_{k}) )
\nonumber \\
& = & -A \dot{x}  \sin(\alpha^x_{k}) = -r_b^{k} \sin(\alpha^x_{k}),
\end{eqnarray}
therefore, $r_b^{k+1} = r_b^{k} (1-\sin(\alpha^x_{k}))$; this proves (\ref{a}).
Similarly,
\begin{eqnarray}
r_c^{k+1}-r_c^{k} & = & A^{\T} (\lambda^{k+1}-\lambda^{k}) + (s^{k+1}-s^{k}) \nonumber \\
& = &  A^{\T} (-\dot{\lambda} \sin(\alpha^s_{k})+\ddot{\lambda}(1-\cos(\alpha^s_{k}) )
-\dot{s} \sin(\alpha^s_{k})+\ddot{s}(1-\cos(\alpha^s_{k}) )
\nonumber \\
& = & -(A^{\T} \dot{\lambda} + \dot{s}) \sin(\alpha^s_{k}) 
   + (A^{\T} \ddot{\lambda} + \ddot{s}) (1-\cos(\alpha^s_{k}))
\nonumber \\
& = & -r_c^{k} \sin(\alpha^s_{k}),
\end{eqnarray}
therefore, $r_c^{k+1} = r_c^{k} (1-\sin(\alpha^s_{k}))$; this proves (\ref{b}).
\hfill \qed
\end{proof}

To show that the duality measure decreases with iterations, we present the following lemma.
\begin{lemma}
Let ${\alpha}_x$ be the step length for $x(\sigma, \alpha)$ and $\alpha_s$ be the step length for 
$s(\sigma, \alpha)$ and $\lambda(\sigma, \alpha)$ defined in Theorem \ref{ellipseSX}. 
Assume that ${\alpha}_x =\alpha_s := \alpha $, then, the updated duality measure can be expressed as 
\begin{eqnarray}
\mu(\alpha) & = & \mu [ 1-\sin(\alpha)+\sigma (1-\cos(\alpha))]
\nonumber \\
& + & \frac{1}{n}\left[ (\ddot{x}^{\T}r_c-\ddot{\lambda}^{\T}r_b ) \sin(\alpha)(1-\cos(\alpha))
+(\dot{x}^{\T}r_c-\dot{\lambda}^{\T}r_b ) (1-\cos(\alpha))^2 \right].
\label{measure}
\end{eqnarray}
\label{simple2}
\end{lemma}
\begin{proof} First, from (\ref{doty}) and (\ref{ddoty}), we have
\[ 
\dot{x}^{\T} A^{\T}\dot{\lambda}-\dot{x}^{\T}\dot{s}=
\dot{x}^{\T} r_c,
\]
this gives
\[
\dot{x}^{\T}\dot{s}=\dot{\lambda}^{\T} r_b -\dot{x}^{\T} r_c.
\]
Similarly,
\[ 
\dot{x}^{\T} \ddot{s} = -\dot{x}^{\T}A^{\T} \ddot{\lambda }=
-\ddot{\lambda }^{\T} r_b, \hspace{0.4in} 
\ddot{x}^{\T} \dot{s} = \ddot{x}^{\T} \dot{s} + \ddot{x}^{\T}A^{\T} \dot{\lambda }=
\ddot{x}^{\T} r_c.
\]
Using these relations with (\ref{duality}) and Lemmas \ref{simple}, we have
\begin{eqnarray}
\mu({\alpha}) 
& = & \left( x-\dot{x}\sin(\alpha_x)+\ddot{x}(1-\cos(\alpha_x)) \right)^{\T} 
\left( s - \dot{s}\sin(\alpha_s)+\ddot{s}(1-\cos(\alpha_s)) \right)/n
\nonumber \\
& = & \frac{x^{\T} s}{n}- \frac{ x^{\T} \dot{s} \sin(\alpha_s)+ s^{\T} \dot{x} \sin(\alpha_x)}{n}
+ \frac{ x^{\T} \ddot{s}(1-\cos(\alpha_s))  + s^{\T} \ddot{x}(1-\cos(\alpha_x)) }{n}
\nonumber \\ 
& & +  \frac{ \dot{x}^{\T}\dot{s}\sin(\alpha_s) \sin(\alpha_x) }{n} 
    - \frac{ \dot{x}^{\T} \ddot{s}\sin(\alpha_x) (1-\cos(\alpha_s))  + \dot{s}^{\T} \ddot{x}\sin(\alpha_s) (1-\cos(\alpha_x)) }{n}
\nonumber \\
& = & \mu [ 1-\sin(\alpha)+\sigma (1-\cos(\alpha))]
+ \frac{ \dot{x}^{\T}\dot{s}\sin^2(\alpha) - 2 \dot{x}^{\T}\dot{s}(1-\cos(\alpha))}{n}
\nonumber \\
& & - \frac{ \dot{x}^{\T} \ddot{s}\sin(\alpha) (1-\cos(\alpha))  + \dot{s}^{\T} \ddot{x}\sin(\alpha) (1-\cos(\alpha)) }{n}
\nonumber \\
& = & \mu [ 1-\sin(\alpha)+\sigma (1-\cos(\alpha))]
\nonumber \\
& + & \frac{1}{n}\left[ (\ddot{x}^{\T}r_c-\ddot{\lambda}^{\T}r_b ) \sin(\alpha)(1-\cos(\alpha))
+(\dot{x}^{\T}r_c-\dot{\lambda}^{\T}r_b ) (1-\cos(\alpha))^2 \right].
\label{diffstepsize}
\end{eqnarray}
This finishes the proof.
\hfill \qed
\end{proof}

The following simple result clearly holds.

\begin{lemma} For $\alpha \in [0, \frac{\pi}{2}]$,
\[
\sin^2(\alpha) =1-\cos^2(\alpha)  \ge 1-\cos(\alpha) \ge \frac{1}{2} \sin^2(\alpha).
\]
\label{sincos}
\hfill \qed
\end{lemma}

\begin{remark}
In view of Lemma \ref{basic}, if $\sin(\alpha)$ is bounded below from zero, then $r_b \rightarrow 0$ 
and $r_c  \rightarrow 0$ as $k \rightarrow \infty$.  
Therefore, in view of Lemmas \ref{simple2} and \ref{sincos}, we have, as $k \rightarrow \infty$,
\[
\mu(\alpha) \approx \mu [ 1-\sin(\alpha)+\sigma (1-\cos(\alpha))] \le \mu [ 1-\sin(\alpha)+\sigma \sin^2(\alpha)]
< \mu
\]
provided that $\dot{\lambda}$, $\dot{x}$, $\ddot{\lambda}$, and $\ddot{x}$ are bounded. This means that equation
$\mu(\alpha) < \mu$ for any $\alpha \in (0, \frac{\pi}{2})$ as $k \rightarrow \infty$. As a matter of fact,
in all numerical test, we have observed the decrease of the duality measure in every iteration
even for ${\alpha}_x \ne \alpha_s$.
\end{remark}

Positivity of $x(\sigma, \alpha_x)$ and $s(\sigma,\alpha_s)$ is guaranteed if $(x,s)>0$ holds and
$\alpha_x$ and $\alpha_s$ are small enough. Assuming that $\dot{x}$, $\dot{s}$, $\ddot{x}$, and $\ddot{s}$
are bounded, the claim can easily be seen from the following relations
\begin{eqnarray}
x(\sigma, \alpha_x) = x-\dot{x} \sin(\alpha_x) +\ddot{x}(1-\cos(\alpha_x)) > 0,
\label{posi1}
\end{eqnarray}
\begin{eqnarray}
s(\sigma, \alpha_s) = s-\dot{s} \sin(\alpha_s) +\ddot{x}(1-\cos(\alpha_s)) > 0.
\label{posi2}
\end{eqnarray}

\section{Implementation details}

In this section, we discuss factors that are normally not discussed in the main body of algorithms but affect 
noticeably, if not significantly, the effectiveness and efficiency of the infeasible interior-point algorithms. 
Most of these factors have been discussed in wide spread literatures, and they are likely implemented differently 
from code to code. We will address all of these implementation topics and provide detailed information of our 
implementation. As we will compare arc-search method and Mehrotra's method, to make a meaningful 
and fair comparison, we will implement everything discussed in this section the same way for both methods, so that the 
only differences of the two algorithms in our implementations are in Steps $4$ and $5$, 
where the arc-search method uses formulae (\ref{update1}) and (\ref{update2}) and Mehrotra's method uses 
(\ref{update1a}) and (\ref{update2a}). But the difference of the computational cost is very small 
because these computations are all analytic.

\subsection{Initial point selection}

Initial point selection has been known an important factor in the computational efficiency for most
infeasible interior-point algorithms. However, many commercial software packages do not 
provide sufficient details, for example, \cite{cmww97, zhang96}. We will use the methods proposed in 
\cite{Mehrotra92,lms92}. We compare the duality measures obtained by these two methods and
select the initial point with smaller duality measure as we guess this selection will reduce 
the number of iterations.

\subsection{Pre-process}

Pre-process or pre-solver is a major factor that can significantly affect the numerical
stability and computational efficiency. Many literatures have been focused on this topic,
for example, \cite{wright97,lms92,bmw75,aa95,mahajan10}. As we will test all linear programming
problems in standard form in Netlib, we focus on the strategies only for the standard 
linear programming problems in the form of (\ref{LP}) and solved in normal equations\footnote{
Some strategies are specifically designed for solving augmented system form, for example, \cite{gosv12}}. 
We will use $A_{i,\cdot}$ for the ith row of A, $A_{\cdot,j}$ for the $j$th 
column of $A$, and $A_{i,j}$ for the element at $(i,j)$ position of $A$. While
reducing the problem, we will express the objective function into two parts,
$c^{\T} x = f_{obj} + \sum_{k} c_k x_k$. The first part $f_{obj}$ at the beginning 
is zero and is updated all the time as we reduce the problem (remove some $c_k$ from $c$); 
the terms in the summation in the
second part are continuously reduced and $c_k$ are updated as necessary when
we reduce the problem. 

The first $6$ pre-process methods presented below were reported in various literatures, such as
\cite{wright97,cmww97,bmw75,aa95,mahajan10}; 
the rest of them, to the best of our knowledge, are not reported anywhere.

\begin{itemize}
\item[] {\bf 1. Empty row} \newline
If $A_{i,\cdot}=0$ and $b_i=0$, this row can be removed. If $A_{i,\cdot}=0$ 
but $b_i \ne 0$, the problem is infeasible.
\item[] {\bf 2. Duplicate rows} \newline
If there is a constant $k$ such that $A_{i,\cdot}=k A_{j,\cdot}$ and $b_i=k b_j$, 
a duplicate row can be removed. If $A_{i,\cdot}=k A_{j,\cdot}$  
but $b_i \ne k b_j$, the problem is infeasible. 
\item[] {\bf 3. Empty column} \newline
If $A_{\cdot,i}=0$ and $c_{i} \ge 0$, $x_i=0$ is the right choice for the minimization,
the $i$th column $A_{\cdot,i}$ and $c_i$ can be removed. If $A_{\cdot,i}=0$ 
but $c_{i} < 0$, the problem is unbounded as $x_i \rightarrow \infty$.
\item[] {\bf 4. Duplicate columns} \newline
If $A_{\cdot, i}=A_{\cdot, j}$, then $A x= b$ can be expressed as
$A_{\cdot, i} (x_i+ x_j)+ \sum_{k \ne i,j} A_{\cdot,k} x_k = b$,
Moreover, if $c_i = c_j$,  $c^{\T} x$ can be expressed as 
$c_i (x_i + x_j)+ \sum_{k \ne i,j} c_k x_k$.
Since $x_i \ge 0$ and $x_j \ge 0$, we have $(x_i + x_j) \ge 0$. Hence, 
a duplicate column can be removed.
\item[] {\bf 5. Row singleton} \newline
If $A_{i,\cdot}$ has exact one nonzero element, i.e., $A_{i,k} \ne 0$ for
some $k$, and for $\forall j \ne k$, $A_{i,j}=0$; then 
$x_k=b_i / A_{i,k}$ and $c^{\T} x= c_k b_i / A_{i,k} + \sum_{j \ne k} c_j x_j$.
For $\ell \ne i$, $A_{\ell,\cdot} x=b_{\ell}$ can be rewritten as
$\sum_{j \ne k} A_{\ell, j} x_j = b_{\ell}-A_{\ell,k} b_i / A_{i,k}$.
This suggests the following update:
(i) if $x_k<0$, the problem is infeasible, otherwise, continue,
(ii) $f_{opt}+ c_k b_i / A_{i,k} \rightarrow f_{opt}$,
(iii) remove $c_k$ from $c$, and (iv) 
$b_{\ell}-A_{\ell,k} b_i / A_{i,k} \rightarrow b_{\ell}$.
With these changes, we can remove the $i$th row and the $k$th column.
\item[] {\bf 6. Free variable} \newline
If $A_{\cdot,i} =- A_{\cdot,j}$ and $c_i =- c_j$, then we can rewrite
$A x=b$ as $A_{\cdot,i} (x_i-x_j) + \sum_{k \ne i,j} A_{\cdot,k} x_k$,
and $c^{\T} x = c_i(x_i-x_j) + \sum_{k \ne i,j} c_k x_k$. The new variable
$x_i-x_j$ is a free variable which can be solved if $A_{\alpha,i} \ne 0$ for
some row $\alpha$ (otherwise, it is an empty column which has been 
discussed). This gives 
\[ 
x_i-x_j= \frac{1}{A_{\alpha,i}} \left(
b_\alpha - \sum_{k \ne i,j} A_{\alpha,k} x_k \right).
\]
For any $A_{\beta,i} \ne 0$, $\beta \ne \alpha$, $A_{\beta,\cdot} x=b_\beta$ can be expressed as
\[
A_{\beta,i}(x_i-x_j) + \sum_{k \ne i,j} A_{\beta, k} x_k = b_\beta,
\] 
or
\[
\frac{A_{\beta,i}}{A_{\alpha,i}} \left(
b_\alpha - \sum_{k \ne i,j} A_{\alpha,k} x_k \right) 
+ \sum_{k \ne i,j} A_{\beta, k} x_k = b_\beta,
\]
or
\[
\sum_{k \ne i,j} \left( A_{\beta, k}- 
\frac{A_{\beta,i}A_{\alpha,k}}{A_{\alpha,i}} \right) x_k 
= b_\beta-\frac{A_{\beta,i} b_{\alpha}}{A_{\alpha,i}}.
\]
Also, $c^{\T} x$ can be rewritten as
\[
\frac{c_i}{A_{\alpha,i}} \left(
b_\alpha - \sum_{k \ne i,j} A_{\alpha,k} x_k \right)
+ \sum_{k \ne i,j} c_k x_k,
\]
or
\[
\frac{c_i b_\alpha}{A_{\alpha,i}} + \sum_{k \ne i,j} \left( 
c_k - \frac{c_i A_{\alpha,k}}{A_{\alpha,i}}  \right) x_k.
\]
This suggests the following update: 
(i) $f_{obj}+\frac{c_i b_\alpha}{A_{\alpha,i}} \rightarrow f_{obj}$,
(ii) $c_k - \frac{c_i A_{\alpha,k}}{A_{\alpha,i}} \rightarrow c_k$, 
(iii) $A_{\beta, k}- \frac{A_{\beta,i}A_{\alpha,k}}{A_{\alpha,i}} \rightarrow A_{\beta, k}$,
(iv) $b_\beta-\frac{A_{\beta,i} b_{\alpha}}{A_{\alpha,i}} \rightarrow b_\beta$, 
(v) delete $A_{\alpha,\cdot}$, $b_\alpha$, $m-1  \rightarrow m$,
delete $A_{\cdot,i}$, $A_{\cdot,j}$, $c_i$, $c_j$, and $n-2 \rightarrow n$.
\item[] {\bf 7. Fixed variable defined by a single row} \newline
If $b_i<0$ and $A_{i,\cdot} \ge 0$ with at least one $j$ such that $A_{i,j}>0$, 
then, the problem is infeasible. Similarly, 
If $b_i>0$ and $A_{i,\cdot} \le 0$ with at least one $j$ such that $A_{i,j}<0$, 
then, the problem is infeasible. 
If $b_i=0$, but either $\max(A_{i,\cdot}) \le 0$ or $\min(A_{i,\cdot}) \ge 0$,
then for any $j$ such that $A_{i,j} \ne 0$, $x_j=0$ has to hold. Therefore, 
we can remove all such rows in $A$ and $b$, and such columns in $A$ and $c$.
\item[] {\bf 8. Fixed variable defined by multiple rows} \newline
If $b_i=b_j$, but either $\max(A_{i,\cdot}-A_{j,\cdot}) \le 0$ or
$\min(A_{i,\cdot}-A_{j,\cdot}) \ge 0$, then for any $k$ such that
$A_{i,k}-A_{j,k} \ne 0$, $x_k=0$ has to hold. This suggests the following update: 
(i) remove $k$th columns of $A$ and $c$ if $A_{i,k}-A_{j,k} \ne 0$, and 
(ii) remove either $i$th or $j$th row depending on which has more nonzeros. 
The same idea can be used for the case when $b_i+b_j=0$.
\item[] {\bf 9. Positive variable defined by signs of $A_{i,\cdot}$ and $b_i$} 
\newline
Since 
\[
x_i = \frac{1}{A_{\alpha,i}} \left(
b_\alpha - \sum_{k \ne i} A_{\alpha,k} x_k \right),
\]
if the sign of $A_{\alpha,i}$ is the same as $b_\alpha$ and opposite to all
$A_{\alpha,k}$ for $k \ne i$, then $x_i \ge 0$ is guaranteed. We can solve
$x_i$, and substitute back into $Ax=b$ and $c^{\T} x$. This suggests taking
the following actions: (i) if $A_{\beta,i} \ne 0$, 
$b_\beta - \frac{A_{\beta,i} b_\alpha}{A_{\alpha,i}} \rightarrow b_\beta$,
(ii) moreover, if $A_{\alpha,k}  \ne 0$, then
$A_{\beta,k} - \frac{A_{\beta,i}A_{\alpha,k}}{A_{\alpha,i}} \rightarrow A_{\beta,k}$,
(iii) $f_{obj}+\frac{c_i b_\alpha}{A_{\alpha,i}}  \rightarrow f_{obj}$,
(iv) $c_k-\frac{c_i A_{\alpha,k}}{A_{\alpha,i}}  \rightarrow c_k$, and
(v) remove the $\alpha$th row and $i$th column.
\item[] {\bf 10. A singleton variable defined by two rows}  \newline
If $A_{i,\cdot}-A_{j,\cdot}$ is a singleton and $A_{i,k}-A_{j,k} \ne 0$ for one and only one $k$,
then $x_k=\frac{b_i-b_j}{A_{i,k}-A_{j,k}}$. This suggests the following update: 
(i) if $x_k \ge 0$ does not hold, the problem is infeasible, (ii) if $x_k \ge 0$ 
does hold, for $\forall \ell \ne i,j$ and $A_{\ell,k} \ne 0$,
$b_\ell-A_{\ell,k}\frac{b_i-b_j}{A_{i,k}-A_{j,k}}  \rightarrow b_\ell$,
(iii) remove either the $i$th or the $j$th row, and remove the $k$th column
from $A$, (iv) remove $c_k$ from $c$, and (v) update 
$f_{obj}+c_k \frac{b_i-b_j}{A_{i,k}} \rightarrow f_{obj}$.
\end{itemize}

We have tested all these ten pre-solvers, and they all work in terms of reducing the problem sizes and
making the problems easier to solve in most cases. But pre-solvers $2,\,\,4,\,\,6,\,\,8$ and $10$ are observed to 
be significantly more time consuming than pre-solvers $1,\,\,3,\,\,5,\,\,7$ and $9$. Moreover, our experience shows
that pre-solvers $1,\,\,3,\,\,5,\,\,7$ and $9$ are more efficient in reducing the problem sizes than 
pre-solvers $2,\,\,4,\,\,6,\,\,8$ and $10$. Therefore, in our implementation, we use only pre-solvers 
$1,\,\,3,\,\,5,\,\,7$ and $9$ for all of our test problems.

\begin{remark}
Our extensive experience (by testing Netlib problems with various combinations of the pre-solves
and comparing results composed of the first five columns of Table 1 
in the next section and the corresponding columns of Table 1 in \cite{cmww97}) shows that 
the set of our pre-process methods uses less time and reduces the 
problem size more efficiently than the set of pre-process methods 
discussed and implemented in \cite{cmww97}.
\end{remark}

\subsection{Matrix scaling}

For ill-conditioned matrix $A$ where the ratio 
$\frac{\max{|A_{i,j}|}}{\min { \{ | A_{k,l} |  A_{k,l} \ne 0 \} } }$ is big, 
scaling is believed to be a good practice, 
for example, see \cite{cmww97}. PCx adopted a scaling strategy 
proposed in \cite{cr72}. 
Let $\Phi=\diag(\phi_1, \cdots, \phi_m)$ and 
$\Psi=\diag(\psi_1, \cdots, \psi_n)$ be the diagonal scaling
matrices of $A$. The scaling for matrix $A$ in \cite{cmww97,cr72} is equivalent to minimize
\[
\sum_{A_{ij} \ne 0} \log^2 \Big| \frac{ A_{ij} }{\phi_i \psi_j } \Big|.
\]
Different methods are proposed to solve this problem \cite{cmww97, cr72}. 
Our extensive experience with these methods and some variations
(by testing all standard problems in Netlib and comparing the results) makes us to believe that although
scaling can improve efficiency and numerical stability of infeasible interior-point algorithms 
for many problems, but over all, it does not help a lot. 
There are no clear criteria on 
what problems may benefit from scaling and what problems may be adversely
affected by scaling. Therefore, 
we decide not to use scaling in all our test problems.

\subsection{Removing row dependency from $A$}

Theoretically, convergence analyses in most existing literatures assume that the matrix $A$ 
is full rank. Practically, row dependency causes some
computational difficulties. However, many real world problems including some problems 
in Netlib have dependent rows. Though using standard Gaussian elimination method can reduce
$A$ into a full rank matrix, the sparse structure of $A$ will be destroyed.
In \cite{andersen95}, Andersen reported an efficient method that removes
row dependency of $A$. The paper also claimed that not only
the numerical stability is improved by the method, 
but the cost of the effort can also be justified. One of the main ideas is to 
identify most independent rows of $A$ in a cheap and easy way and separate these 
independent rows from those that may be dependent. A variation of Andersen's 
method can be summarized as follows.

First, it assumes that all empty rows have been removed by pre-solver.
Second, matrix $A$ often contains many column singletons (the column has only 
one nonzero), for example, slack variables are column singletons. Clearly, a row 
containing a column singleton cannot be dependent. If these rows
are separated (temporarily removed) from rest rows of $A$, new column 
singletons may appear and more rows may be separated. This process may separate 
most rows from rest rows of $A$ in practice. Permutation operations can be used 
to move the singletons to the diagonal elements of $A$. The dependent rows are 
among the rows left in the process. Then, Gaussian elimination method can be 
applied with pivot selection using Markowitz criterion \cite{duff89,dobes05}. 
Some implementation details include (a) break ties by choosing element with 
the largest magnitude, and (b) use threshold pivoting.

Our extensive experience makes us to believe that although Andersen's method may be worthwhile
for some problems and significantly improve the numerical stability, but it may be expensive
for many other problems. We choose to not use this function unless we feel it is necessary 
when it is used as part of handling degenerate solutions discussed later. To have a
fair comparison between two algorithms, we will make it clear in our test report what algorithms 
and/or problems use this function and what algorithms and/or problems do not use this function.

\subsection{Linear algebra for sparse Cholesky matrix}

Similar to Mehrotra's algorithm, the majority of the computational cost of our proposed algorithm is
to solve sparse Cholesky systems (\ref{doy2}) and (\ref{ddoy2}), which
can be expressed as an abstract problem as follows. 
\begin{equation}
AD^2A^{\T} u =v,
\label{useLater}
\end{equation}
where $D=X^{\frac{1}{2}}S^{-\frac{1}{2}}$ is identical in (\ref{doy2}) and (\ref{ddoy2}), 
but $u$ and $v$ are different vectors. Many popular LP solvers \cite{cmww97,zhang96} call a 
software package \cite{ np93} which uses some linear algebra specifically developed for 
the sparse Cholesky decomposition \cite{liu85}. However, MATLAB does not yet have
this function to call. This is the major difference of our implementation comparing to other
popular LP solvers, which is most likely the main reason that our test results are
slightly different from test results reported in other literatures.

\subsection{Handling degenerate solutions}

An important result in linear programming \cite{gt56} is that there always exist strictly 
complementary optimal solutions which meet the conditions $x^*  \circ s^*=0$ and $ x^*  +  s^*>0$.
Therefore, the columns of $A$ can be partitioned as $B \subseteq \{ 1, 2, \ldots, n\}$, 
the set of indices of the positive coordinates of $x^*$, and $N \subseteq \{ 1, 2, \ldots, n\}$, 
the set of indices of the positive coordinates of $s^*$, such that $B \cup N =\{ 1, 2, \ldots, n\}$ 
and $ B \cap N =\emptyset$. Thus, we can partition $A=(A_B, A_N)$, and define 
the primal and dual optimal faces by
\[
{\cal P}_*=\{ x: A_Bx_B=b, x \ge 0, x_N=0 \},
\]
and
\[
{\cal D}_*=\{ (\lambda, s): A^{\T}_N \lambda +S_N =c_N, s_B =0, s \ge 0 \}.
\]
However, not all optimal solutions in linear programming are strictly complementary. A simple
example is provided in \cite[Page 28]{wright97}. 
Although many interior-point algorithms are proved to converge strictly to complementary
solutions, this claim may not be true for Mehrotra's method and arc-search method proposed in this paper.

Recall that the problem pair (1) and (2) is called to have a primal degenerate solution
if a primal optimal solution $x^*$ has less than $m$ positive coordinates, and have a dual 
degenerate solution if a dual optimal solution $s^*$ has less than $n-m$ positive coordinates. 
The pair $(x^*,s^*)$ is called degenerate if it is primal or dual degenerate.
This means that as $x^k \rightarrow x^*$, equation (\ref{useLater}) can be written as 
\begin{equation}
(A_BX_BS_B^{-1}A^{\T}_B) u =v,
\label{useNow}
\end{equation}
If the problem converges to a primal degenerate solution, then the rank of $(A_BX_BS_B^{-1}A^{\T}_B)$ 
is less than $m$ as $x^k \rightarrow x^*$. In this case, there is a
difficulty to solve (\ref{useNow}).
Difficulty caused by degenerate solutions in interior-point methods for linear programming 
has been realized for a long time
\cite{ghrt93}. We have observed this troublesome incidence in quite a few Netlib test problems. 
Similar observation was also reported in \cite{gmstw86}.
Though we don't see any special attention or report on this troublesome issue from some widely
cited papers and LP solvers, such as \cite{Mehrotra92,lms91,lms92,cmww97,zhang96}, we noticed from
\cite[page 219]{wright97} that 
some LP solvers \cite{cmww97,zhang96} twisted the sparse Cholesky decomposition code
\cite{np93} to overcome the difficulty.

In our implementation, we use a different method to avoid the difficulty
because we do not have access to the code of \cite{np93}. After each iteration, minimum $x^k$ is examined. If
$\min \{ x^k \} \le \epsilon_x$, then, for all components of $x$ satisfying $x_i \le \epsilon_x$,
we delete $A_{.i}$, $x_i$, $s_i$, $c_i$, and the $i$th component
of $r_c$; use the 
method proposed in Subsection 4.4 to check if the updated $A$ is
full rank and make the updated $A$ full rank if it is necessary.

The default $\epsilon_x$ is $10^{-6}$. For problems that needs a different $\epsilon_x$,
we will make it clear in the report of the test results.

\subsection{Analytic solution of $\alpha^x$ and $\alpha^s$}

We know that $\alpha^x$ and $\alpha^s$ in (\ref{update1a}) can easily 
be calculated in analytic form. Similarly, $\alpha^x$ and $\alpha^s$ 
in (\ref{update1}) can also be calculated in analytic form as follows. 
For each $i \in \lbrace 1,\ldots, n \rbrace$, we can select the largest
$\alpha_{x_i}$ such that for any $\alpha \in [0, \alpha_{x_i}]$, the 
$i$th inequality of (\ref{updatex}) holds, and the largest $\alpha_{s_i}$ 
such that for any $\alpha \in [0, \alpha_{s_i}]$ the $i$th inequality 
of (\ref{updates}) holds. We then define 
\begin{eqnarray}
{\alpha^x}=\min_{i \in \lbrace 1,\ldots, n \rbrace}
\lbrace \alpha_{x_i}\rbrace, 
\\
{\alpha^s}=\min_{i \in \lbrace 1,\ldots, n \rbrace}
\lbrace \alpha_{s_i} \rbrace.
\label{alpha}
\end{eqnarray}
$\alpha_{x_i}$ and $\alpha_{s_i}$ can be given in analytical forms 
according to the values of $\dot{x}_i$, $\ddot{x}_i$, $\dot{s}_i$, 
$\ddot{s}_i$. First, from (\ref{update1}), we have
\begin{equation}
x_i +\ddot{x}_i
\ge \dot{x}_i\sin(\alpha)+\ddot{x}_i\cos(\alpha).
\label{alphai}
\end{equation}
Clearly, let $\beta=\sin(\alpha)$, this is equivalent to finding $\beta \in (0,1]$
such that 
\begin{equation}
x_i 
-\dot{x}_i\beta+\ddot{x}_i(1-\sqrt{1-\beta^2}) \ge 0.
\label{betai}
\end{equation}
But we prefer to use (\ref{alphai}) in the following analysis because of its geometric property.

\vspace{0.06in}
\noindent{\it Case 1 ($\dot{x}_i=0$ and $\ddot{x}_i\ne 0$)}:

For $\ddot{x}_i \ge -x_i$, and for any 
$\alpha \in [0, \frac{\pi}{2}]$, $x_i(\alpha) \ge 0$ holds.
For $\ddot{x}_i \le -x_i$, to meet (\ref{alphai}),
we must have 
$\cos(\alpha) \ge \frac{x_i +\ddot{x}_i}{\ddot{x}_i}$, or,
$\alpha \le \cos^{-1}\left( \frac{x_i +\ddot{x}_i}
{\ddot{x}_i} \right)$. Therefore,
\begin{equation}
\alpha_{x_i} = \left\{
\begin{array}{ll}
\frac{\pi}{2} & \quad \mbox{if $x_i +\ddot{x}_i \ge 0$} \\
\cos^{-1}\left( \frac{x_i +\ddot{x}_i}
{\ddot{x}_i} \right) & \quad 
\mbox{if $x_i +\ddot{x}_i \le 0$}.
\end{array}
\right.
\label{case1a}
\end{equation}
\noindent{\it Case 2 ($\ddot{x}_i=0$ and $\dot{x}_i\ne 0$)}:

For $\dot{x}_i \le x_i $, and for any 
$\alpha \in [0, \frac{\pi}{2}]$, $x_i(\alpha) \ge 0$ holds.
For $\dot{x}_i \ge x_i $,  to meet (\ref{alphai}),
we must have 
$\sin(\alpha) \le \frac{x_i }{\dot{x}_i}$, or
$\alpha \le \sin^{-1}\left( \frac{x_i }
{\dot{x}_i} \right)$. Therefore,
\begin{equation}
\alpha_{x_i} = \left\{
\begin{array}{ll}
\frac{\pi}{2} & \quad \mbox{if $\dot{x}_i \le x_i $} \\
\sin^{-1}\left( \frac{x_i }
{\dot{x}_i} \right) & \quad \mbox{if $\dot{x}_i \ge x_i $}
\end{array}
\right.
\label{case2a}
\end{equation}
\noindent{\it Case 3 ($\dot{x}_i>0$ and $\ddot{x}_i>0$)}:

Let $\dot{x}_i=\sqrt{\dot{x}_i^2+\ddot{x}_i^2}\cos(\beta)$, and
$\ddot{x}_i=\sqrt{\dot{x}_i^2+\ddot{x}_i^2}\sin(\beta)$, (\ref{alphai})
can be rewritten as 
\begin{equation}
x_i + \ddot{x}_i \ge \sqrt{\dot{x}_i^2+\ddot{x}_i^2}
\sin(\alpha + \beta),
\label{alphai1}
\end{equation}
where 
\begin{equation}
\beta = \sin^{-1} \left( \frac{\ddot{x}_i }
{\sqrt{\dot{x}_i^2+\ddot{x}_i^2}} \right).
\label{beta1}
\end{equation}
For $\ddot{x}_i + x_i \ge \sqrt{\dot{x}_i^2+\ddot{x}_i^2}$, and 
for any $\alpha \in [0, \frac{\pi}{2}]$, $x_i(\alpha) \ge 0$ holds.
For $\ddot{x}_i + x_i \le \sqrt{\dot{x}_i^2+\ddot{x}_i^2}$,  
to meet (\ref{alphai1}), we must have 
$\sin(\alpha + \beta) \le \frac{x_i + \ddot{x}_i}
{\sqrt{\dot{x}_i^2+\ddot{x}_i^2}}$, or
$\alpha + \beta \le \sin^{-1}\left( \frac{x_i + \ddot{x}_i }
{\sqrt{\dot{x}_i^2+\ddot{x}_i^2}} \right)$. Therefore,
\begin{equation}
\alpha_{x_i} = \left\{
\begin{array}{ll}
\frac{\pi}{2} & \quad \mbox{if $x_i + \ddot{x}_i \ge 
\sqrt{\dot{x}_i^2+\ddot{x}_i^2}$} \\
\sin^{-1}\left( \frac{x_i + \ddot{x}_i }
{\sqrt{\dot{x}_i^2+\ddot{x}_i^2}} \right) - \sin^{-1}\left( \frac{\ddot{x}_i } 
{\sqrt{\dot{x}_i^2+\ddot{x}_i^2}} \right) & \quad 
\mbox{if $x_i + \ddot{x}_i \le 
\sqrt{\dot{x}_i^2+\ddot{x}_i^2}$}
\end{array}
\right.
\label{case3a}
\end{equation}
\noindent{\it Case 4 ($\dot{x}_i>0$ and $\ddot{x}_i<0$)}:

Let $\dot{x}_i=\sqrt{\dot{x}_i^2+\ddot{x}_i^2}\cos(\beta)$, and
$\ddot{x}_i=-\sqrt{\dot{x}_i^2+\ddot{x}_i^2}\sin(\beta)$, (\ref{alphai})
can be rewritten as 
\begin{equation}
x_i + \ddot{x}_i \ge \sqrt{\dot{x}_i^2+\ddot{x}_i^2}
\sin(\alpha - \beta),
\label{alphai2}
\end{equation}
where 
\begin{equation}
\beta = \sin^{-1} \left( \frac{-\ddot{x}_i }
{\sqrt{\dot{x}_i^2+\ddot{x}_i^2}} \right).
\label{beta2}
\end{equation}
For $\ddot{x}_i + x_i \ge \sqrt{\dot{x}_i^2+\ddot{x}_i^2}$, and 
for any $\alpha \in [0, \frac{\pi}{2}]$, $x_i(\alpha) \ge 0$ holds.
For $\ddot{x}_i + x_i \le \sqrt{\dot{x}_i^2+\ddot{x}_i^2}$,  
to meet (\ref{alphai2}), we must have 
$\sin(\alpha - \beta) \le \frac{x_i + \ddot{x}_i}
{\sqrt{\dot{x}_i^2+\ddot{x}_i^2}}$, or
$\alpha - \beta \le \sin^{-1}\left( \frac{x_i + \ddot{x}_i }
{\sqrt{\dot{x}_i^2+\ddot{x}_i^2}} \right)$. Therefore,
\begin{equation}
\alpha_{x_i} = \left\{
\begin{array}{ll}
\frac{\pi}{2} & \quad \mbox{if $x_i + \ddot{x}_i \ge 
\sqrt{\dot{x}_i^2+\ddot{x}_i^2}$} \\
\sin^{-1}\left( \frac{x_i + \ddot{x}_i }
{\sqrt{\dot{x}_i^2+\ddot{x}_i^2}} \right) + \sin^{-1}\left( \frac{-\ddot{x}_i }
{\sqrt{\dot{x}_i^2+\ddot{x}_i^2}} \right) & \quad 
\mbox{if $x_i + \ddot{x}_i \le 
\sqrt{\dot{x}_i^2+\ddot{x}_i^2}$}
\end{array}
\right.
\label{case4a}
\end{equation}
\noindent{\it Case 5 ($\dot{x}_i<0$ and $\ddot{x}_i<0$)}:

Let $\dot{x}_i=-\sqrt{\dot{x}_i^2+\ddot{x}_i^2}\cos(\beta)$, and
$\ddot{x}_i=-\sqrt{\dot{x}_i^2+\ddot{x}_i^2}\sin(\beta)$, (\ref{alphai})
can be rewritten as 
\begin{equation}
x_i + \ddot{x}_i \ge -\sqrt{\dot{x}_i^2+\ddot{x}_i^2}
\sin(\alpha +\beta),
\label{alphai4}
\end{equation}
where 
\begin{equation}
\beta = \sin^{-1} \left( \frac{-\ddot{x}_i }
{\sqrt{\dot{x}_i^2+\ddot{x}_i^2}} \right).
\label{beta3}
\end{equation}
For $\ddot{x}_i + x_i\ge 0$ and 
any $\alpha \in [0, \frac{\pi}{2}]$, $x_i(\alpha) \ge 0$ holds.
For $\ddot{x}_i + x_i\le 0$,  
to meet (\ref{alphai4}), we must have 
$\sin(\alpha + \beta) \ge \frac{-(x_i + \ddot{x}_i)}
{\sqrt{\dot{x}_i^2+\ddot{x}_i^2}}$, or
$\alpha + \beta \le \pi - \sin^{-1} \left( \frac{-(x_i + 
\ddot{x}_i)}{\sqrt{\dot{x}_i^2+\ddot{x}_i^2}} \right)$. Therefore,
\begin{equation}
\alpha_{x_i} = \left\{
\begin{array}{ll}
\frac{\pi}{2} & \quad \mbox{if $x_i + \ddot{x}_i \ge 0$} \\
\pi - \sin^{-1} \left( \frac{-(x_i + \ddot{x}_i) }
{\sqrt{\dot{x}_i^2+\ddot{x}_i^2}} \right) - \sin^{-1}\left( \frac{-\ddot{x}_i }
{\sqrt{\dot{x}_i^2+\ddot{x}_i^2}} \right) 
& \quad \mbox{if $x_i + \ddot{x}_i \le 0$}
\end{array}
\right.
\label{case5a}
\end{equation}
\noindent{\it Case 6 ($\dot{x}_i<0$ and $\ddot{x}_i>0$)}:

Clearly (\ref{alphai}) always holds for $\alpha \in [0, \frac{\pi}{2}]$. 
Therefore, we can take \begin{equation}
\alpha_{x_i} = \frac{\pi}{2}.
\end{equation}
\noindent{\it Case 7 ($\dot{x}_i=0$ and $\ddot{x}_i=0$)}:

Clearly (\ref{alphai}) always holds for $\alpha \in [0, \frac{\pi}{2}]$. 
Therefore, we can take \begin{equation}
\alpha_{x_i} = \frac{\pi}{2}.
\end{equation}
Similar analysis can be performed for $\alpha^s$ in (\ref{update1}) and similar results 
can be obtained for $\alpha_{s_i}$. For completeness, we list the formulae
without repeating the proofs.

\noindent{\it Case 1a ($\dot{s}_i=0$, $\ddot{s}_i \ne 0$)}:

\begin{equation}
\alpha_{s_i} = \left\{
\begin{array}{ll}
\frac{\pi}{2} & \quad \mbox{if $s_i +\ddot{s}_i \ge 0$} \\
\cos^{-1}\left( \frac{s_i +\ddot{s}_i}
{\ddot{s}_i} \right) & \quad 
\mbox{if $s_i +\ddot{s}_i \le 0$}.
\end{array}
\right.
\label{case1b}
\end{equation}
\noindent{\it Case 2a ($\ddot{s}_i=0$ and $\dot{s}_i \ne 0$)}:

\begin{equation}
\alpha_{s_i} = \left\{
\begin{array}{ll}
\frac{\pi}{2} & \quad \mbox{if $\dot{s}_i \le s_i $} \\
\sin^{-1}\left( \frac{s_i }
{\dot{s}_i} \right) & \quad \mbox{if $\dot{s}_i \ge s_i $}
\end{array}
\right.
\label{case2b}
\end{equation}
\noindent{\it Case 3a ($\dot{s}_i>0$ and $\ddot{s}_i>0$)}:

\begin{equation}
\alpha_{s_i} = \left\{
\begin{array}{ll}
\frac{\pi}{2} & \quad \mbox{if $s_i + \ddot{s}_i \ge 
\sqrt{\dot{s}_i^2+\ddot{s}_i^2}$} \\
\sin^{-1}\left( \frac{s_i + \ddot{s}_i }
{\sqrt{\dot{s}_i^2+\ddot{s}_i^2}} \right) - \sin^{-1}\left( \frac{\ddot{s}_i } 
{\sqrt{\dot{s}_i^2+\ddot{s}_i^2}} \right) & \quad 
\mbox{if $s_i + \ddot{s}_i < 
\sqrt{\dot{s}_i^2+\ddot{s}_i^2}$}
\end{array}
\right.
\label{case3b}
\end{equation}
\noindent{\it Case 4a ($\dot{s}_i>0$ and $\ddot{s}_i<0$)}:

\begin{equation}
\alpha_{s_i} = \left\{
\begin{array}{ll}
\frac{\pi}{2} & \quad \mbox{if $s_i + \ddot{s}_i \ge 
\sqrt{\dot{s}_i^2+\ddot{s}_i^2}$} \\
\sin^{-1}\left( \frac{s_i + \ddot{s}_i }
{\sqrt{\dot{s}_i^2+\ddot{s}_i^2}} \right) + \sin^{-1}\left( \frac{-\ddot{s}_i }
{\sqrt{\dot{s}_i^2+\ddot{s}_i^2}} \right) & \quad 
\mbox{if $s_i + \ddot{s}_i \le 
\sqrt{\dot{s}_i^2+\ddot{s}_i^2}$}
\end{array}
\right.
\label{case4b}
\end{equation}
\noindent{\it Case 5a ($\dot{s}_i<0$ and $\ddot{s}_i<0$)}:

\begin{equation}
\alpha_{s_i} = \left\{
\begin{array}{ll}
\frac{\pi}{2} & \quad \mbox{if $s_i + \ddot{s}_i \ge 0$} \\
\pi - \sin^{-1} \left( \frac{-(s_i + \ddot{s}_i) }
{\sqrt{\dot{s}_i^2+\ddot{s}_i^2}} \right) - \sin^{-1}\left( \frac{-\ddot{s}_i }
{\sqrt{\dot{s}_i^2+\ddot{s}_i^2}} \right) 
& \quad \mbox{if $s_i + \ddot{s}_i \le 0$}
\end{array}
\right.
\label{case5b}
\end{equation}
\noindent{\it Case 6a ($\dot{s}_i<0$ and $\ddot{s}_i>0$)}:

Clearly (\ref{alphai}) always holds for $\alpha \in [0, \frac{\pi}{2}]$.
Therefore, we can take
\begin{equation}
\alpha_{s_i} = \frac{\pi}{2}.
\end{equation}
\noindent{\it Case 7a ($\dot{s}_i=0$ and $\ddot{s}_i=0$)}:

Clearly (\ref{alphai}) always holds for $\alpha \in [0, \frac{\pi}{2}]$. 
Therefore, we can take \begin{equation}
\alpha_{s_i} = \frac{\pi}{2}.
\end{equation}

\subsection{Step scaling parameter}

A fixed step scaling parameter is used in PCx \cite{cmww97}. A more sophisticated step scaling 
parameter is used in LIPSOL according to \cite[Pages 204-205]{wright97}. In our implementation, 
we use an adaptive step scaling parameter which is given below
\begin{equation}
\beta = 1-e^{-(k+2)},
\label{beta}
\end{equation}
where $k$ is the number of iterations. This parameter will approach to one as $k \rightarrow \infty$.

\subsection{Terminate criteria}

The main stopping criterion used in our implementations of arc-search method and
Mehrotra's method is similar to that of LIPSOL \cite{zhang96}
\[
\frac{\|r_b^k\|}{\max \lbrace 1, \| b\| \rbrace }
+\frac{\|r_c^k \|}{\max \lbrace 1, \| c\|  \rbrace }
+\frac{ \mu_k }{\max \lbrace 1, \| c^{\T}x^k \|, \|b^{\T}\lambda^k \|  \rbrace } 
< 10^{-8}.
\] 

In case that the algorithms fail to find a good search direction,
the programs also stop if step sizes $\alpha_k^x < 10^{-8}$ and
$\alpha_k^s < 10^{-8}$. 

Finally, if due to the numerical problem, $r_b^{k}$ or $r_c^{k}$ does not decrase but 
$10r_b^{k-1}<r_b^{k}$ or $10r_c^{k-1}<r_c^{k}$, the programs stop.

\section{Numerical Tests}

In this section, we first examine a simple problem and show graphically what feasible 
central path and infeasible central path look like, why ellipsoidal approximation may be a 
better approximation to infeasible central path than a straight line, and how 
arc-search is carried out for this simple problem. Using a plot, 
we can easily see that searching along the ellipse is more attractive than searching along a 
straight line. We then provide the numerical test results of larger scale Netlib test problems to 
validate our observation from this simple problem.

\subsection{A simple illustrative example}

Let us consider 
\begin{example}
\[
\min x_1, \hspace{0.15in} s.t. \hspace{0.1in} x_1+x_2 = 5,
\hspace{0.1in} x_1 \ge 0, \hspace{0.1in} x_2 \ge 0. 
\]
\end{example}
The feasible central path $(x,s)$ defined in (\ref{centralpath}) satisfies the following conditions:
\[
x_1+x_2 = 5,
\]
\[
\left[ \begin{array}{c} 1 \\ 1 \end{array} \right] \lambda + 
\left[ \begin{array}{c} s_1 \\ s_2 \end{array} \right] =
\left[ \begin{array}{c} 1 \\ 0 \end{array} \right],
\]
\[
x_1s_1 =\mu, \hspace{.15in} x_2s_2 =\mu.
\]
The optimizer is given by $x_1=0$, $x_2=5$, $\lambda=0$, $s_1=1$, and $s_2=0$.
The feasible central path of this problem is given analytically as
\begin{subequations}
\begin{gather}
\lambda=\frac{5-2\mu-\sqrt{(5-2\mu)^2+20\mu}}{10},
\\
s_1 = 1-\lambda, \hspace{0.1in} s_2 = -\lambda, \hspace{0.1in} 
x_1s_1 =\mu, \hspace{.1in} x_2s_2 =\mu.
\end{gather}
\label{simpleEx}
\end{subequations}

\begin{figure}[ht]
\centerline{\epsfig{file=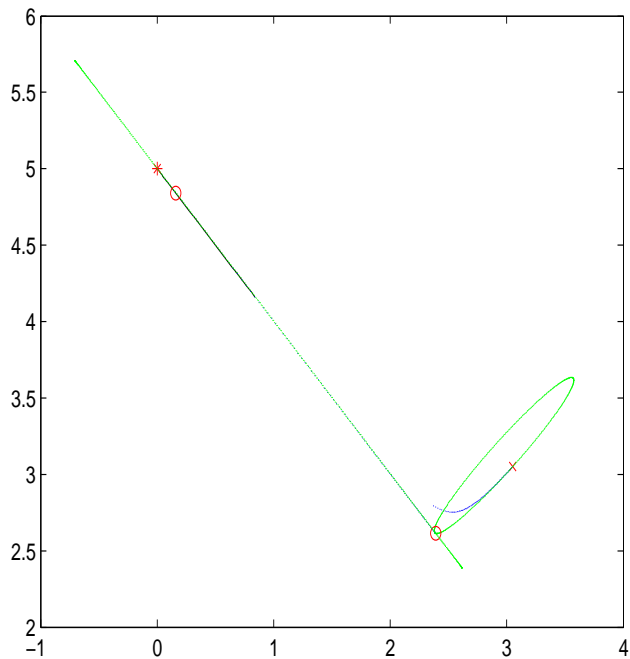,height=10cm,width=10cm}}
\caption{Arc-search for the simple example.}
\label{fig:iter1}
\end{figure}

The feasible and infeasible central paths are arcs in $5$-dimensional space 
$(\lambda, x_1, s_1, x_2, s_2)$. 
If we project the central paths into $2$-dimensional subspace spanned by 
$(x_1,x_2)$, they are arcs in $2$-dimensional subspace. Figure \ref{fig:iter1} shows 
the first two iterations of Algorithm \ref{mainAlgo3} in the $2$-dimensional
subspace spanned by $(x_1,x_2)$. In Figure \ref{fig:iter1}, 
the initial point $(x_1^0, x_2^0)$ is marked by 'x' in red;
the optimal solution is marked by '*' in red;
$(\dot{x}, \dot{s},
\dot{\lambda})$ is calculated by using (\ref{doty}); 
$(\ddot{x}, \ddot{s}, \ddot{\lambda})$ is calculated by using 
(\ref{ddoty}); the projected feasible central path ${\cal C}(t)$ near the optimal solution is 
calculated by using (\ref{simpleEx}) and
is plotted as a continuous line 
in black; the infeasible central path ${\cal H}(t)$ starting from current iterate is calculated by using 
(\ref{neiborArc}) and plotted as the dotted lines in blue; and
the projected ellipsoidal approximations ${\cal E}(\alpha)$ are the dotted 
lines in green (they may look like continuous line 
some times because many dots are used).
In the first iteration, the iterate 'x' moves along the ellipse (defined by in Theorem \ref{ellipseSX})
to reach the next iterate marked as 'o' in red because the calculation of infeasible central path 
(the blue line) is very expensive and ellipse is cheap to calculate and a better approximation to 
the infeasible central path than a straight line. The rest iterations are simply the repetition of 
the process until it reaches the optimal solution $(s^*, x^*)$. Only two iterations 
are plotted in Figure \ref{fig:iter1}.

It is worthwhile to note that in this simple problem, the infeasible central path has a sharp turn
in the first iteration which may happen a number of times for general problem as discussed
in \cite{vy96}. The arc-search method is expected to perform better than
Mehrotra's method in iterations that are 
close to the sharp turns. In this simple problem, after the first iteration, the feasible central
path ${\cal C}(t)$, the infeasible central path ${\cal H}(t)$, and the ellipse ${\cal E}(\alpha)$ 
are all very close to each other and close to a straight line.

\subsection{Netlib test examples}

The algorithm developed in this paper is implemented in a Matlab
function.
Mehrotra's algorithm is also implemented in a Matlab function.
They are almost identical. Both algorithms use exactly the same initial point, the same
stopping criteria, the same pre-process, and the same parameters.
The only difference of the two implementations is that arc-search method
searches optimizer along an ellipse and Mehrotra's method searches
optimizer along a straight line.
Numerical tests for both algorithms have been performed for 
all Netlib LP problems that are presented in standard form. 
The iteration numbers used to solve these problems are listed in  
Table 1. Only one Netlib problem {\tt Osa\_60} ($m=10281$ and $n=232966$) presented in standard form is 
not included in the test because the PC computer used for the testing
does not have enough memory to handle this problem.

\footnotesize
\begin{longtable}{|c|c|c|c|c|c|r|c|c|r|c|}
\hline          
\multirow{2}{*}{Problem} & \multicolumn{2}{c|}{before prep} & \multicolumn{2}{c|}{after prep} & \multicolumn{3}{|c|}{Arc-search}  & \multicolumn{3}{|c|}{Mehrotra} \\ \cline{2-11}
           & m & n & m & n & iter  & objective & infeas & iter  & objective & infeas \\ 
\hline
Adlittle     & 56   & 138   & 54   & 136   & 15  &   2.2549e+05   & 1.0e-07 & 15 & 2.2549e+05  & 3.4e-08 \\ \hline
Afiro        & 27   & 51    & 8    & 32    & 9   &   -464.7531    & 1.0e-11 &  9 & -464.7531   & 8.0e-12 \\ \hline
Agg          & 488  & 615   & 391  & 479   & 18  &   -3.5992e+07  & 5.0e-06 & 22 & -3.5992e+07 & 5.2e-05 \\ \hline
Agg2         & 516  & 758   & 514  & 755   & 18  &   -2.0239e+07  & 4.6e-07 & 20 & -2.0239e+07 & 5.2e-07 \\ \hline
Agg3         & 516  & 758   & 514  & 755   & 17  &   1.0312e+07   & 3.1e-08 & 18 & 1.0312e+07  & 8.8e-09 \\ \hline
Bandm        & 305  & 472   & 192  & 347   & 19  &   -158.6280    & 3.2e-11 & 22 & -158.6280   & 8.3e-10 \\ \hline
Beaconfd     & 173  & 295   & 57   & 147   & 10  &   3.3592e+04   & 1.4e-12 & 11 & 3.3592e+04  & 1.4e-10 \\ \hline
Blend        & 74   & 114   & 49   & 89    & 12  &   -30.8121     & 1.0e-09 & 14 & -30.8122    & 4.9e-11 \\ \hline
Bnl1         & 643  & 1586  & 429  & 1314  & 32  &   1.9776e+03   & 2.7e-09 & 35 & 1.9776e+03  & 3.4e-09 \\ \hline
Bnl2         & 2324 & 4486  & 1007 & 3066  & 32  &   1.8112e+03   & 5.4e-10 & 37 & 1.8112e+03  & 9.3e-07 \\ \hline
Brandy       & 220  & 303   & 113  & 218   & 20  &   1.5185e+03   & 3.0e-06 & 19 & 1.5185e+03  & 6.2e-08 \\ \hline
Degen2*      & 444  & 757   & 440  & 753   & 16  &   -1.4352e+03  & 1.9e-08 & 17 & -1.4352e+03 & 2.0e-10 \\ \hline
Degen3*      & 1503 & 2604  & 1490 & 2591  & 22  &   -9.8729e+02  & 7.0e-05  & 22 & -9.8729e+02 & 1.2e-09 \\ \hline
fffff800     & 525  & 1208  & 487  & 991   & 26  &   5.5568e+005  & 4.3e-05 & 31 & 5.5568e+05  & 7.7e-04 \\ \hline
Israel       & 174  & 316   & 174  & 316   & 23  &   -8.9664e+05  & 7.4e-08 & 29 & -8.9665e+05 & 1.8e-08 \\ \hline
Lotfi        & 153  & 366   & 113  & 326   & 14  &   -25.2647     & 3.5e-10 & 18 & -25.2647    & 2.7e-07 \\ \hline
Maros\_r7    & 3136 & 9408  & 2152 & 7440  & 18  &   1.4972e+06   & 1.6e-08 & 21 & 1.4972e+06  & 6.4e-09 \\ \hline
Osa\_07*     & 1118 & 25067 & 1081 & 25030 & 37  &   5.3574e+05   & 4.2e-07 & 35 & 5.3578e+05  & 1.5e-07 \\ \hline
Osa\_14      & 2337 & 54797 & 2300 & 54760 & 35  &  1.1065e+06    & 2.0e-09 & 37 & 1.1065e+06  & 3.0e-08 \\ \hline
Osa\_30      & 4350 &104374 & 4313 & 104337& 32  &  2.1421e+06    & 1.0e-08 & 36 &  2.1421e+06 & 1.3e-08 \\ \hline
Qap12        & 3192 & 8856  & 3048 & 8712  & 22  &  5.2289e+02    & 1.9e-08 & 24 & 5.2289e+02  & 6.2e-09  \\ \hline
Qap15*       & 6330 & 22275 & 6105 & 22050 & 27  &  1.0411e+03    & 3.9e-07 & 44 & 1.0410e+03  & 1.5e-05  \\ \hline
Qap8*        & 912  & 1632  & 848  & 1568  & 12  &   2.0350e+02   & 1.2e-12 & 13 & 2.0350e+02  & 7.1e-09  \\ \hline
Sc105        & 105  & 163   & 44   & 102   & 10  &   -52.2021     & 3.8e-12 & 11 & -52.2021    & 9.8e-11 \\ \hline
Sc205        & 205  & 317   & 89   & 201   & 13  &   -52.2021     & 3.7e-10 & 12 & -52.2021    & 8.8e-11 \\ \hline
Sc50a        & 50   & 78    & 19   & 47    & 10  &   -64.5751     & 3.4e-12 & 9  & -64.5751    & 8.3e-08 \\ \hline
Sc50b        & 50   & 78    & 14   & 42    & 8   &   -70.0000     & 1.0e-10 & 8  & -70.0000    & 9.1e-07 \\ \hline
Scagr25      & 471  & 671   & 343  & 543   & 19  &  -1.4753e+07   & 5.0e-07 & 18 & -1.4753e+07 & 4.6e-09 \\ \hline
Scagr7       & 129  & 185   & 91   & 147   & 15  &   -2.3314e+06  & 2.7e-09 & 17 & -2.3314e+06 & 1.1e-07 \\ \hline
Scfxm1+      & 330  & 600   & 238  & 500   & 20  &   1.8417e+04   & 3.1e-07 & 21 &  1.8417e+04 & 1.6e-08 \\ \hline
Scfxm2       & 660  & 1200  & 479  & 1003  & 23  &   3.6660e+04   & 2.3e-06 & 26 &  3.6660e+04 & 2.6e-08 \\ \hline
Scfxm3+      & 990  & 1800  & 720  & 1506  & 24  &   5.4901e+04   & 1.9e-06 & 23 &  5.4901e+04 & 9.8e-08 \\ \hline
Scrs8        & 490  & 1275  & 115  & 893   & 23  &   9.0430e+02   & 1.2e-11 & 30 &  9.0430e+02 & 1.8e-10 \\ \hline
Scsd1        & 77   & 760   & 77   & 760   & 12  &   8.6666       & 1.0e-10 & 13 &     8.6666  & 8.7e-14 \\ \hline
Scsd6        & 147  & 1350  & 147  & 1350  & 14  &   50.5000      & 1.5e-13 & 16 &     50.5000 & 7.9e-13 \\ \hline
Scsd8        & 397  & 2750  & 397  & 2750  & 13  &   9.0500e+02   & 6.7e-10 & 14 &  9.0500e+02 & 1.3e-10 \\ \hline
Sctap1       & 300  & 660   & 284  & 644   & 20  &   1.4122e+03   & 2.6e-10 & 24 &  1.4123e+03 & 2.1e-09 \\ \hline
Sctap2       & 1090 & 2500  & 1033 & 2443  & 20  &   1.7248e+03   & 2.1e-10 & 21 &  1.7248e+03 & 4.4e-07 \\ \hline
Sctap3       & 1480 & 3340  & 1408 & 3268  & 20  &   1.4240e+03   & 5.7e-08 & 22 &  1.4240e+03 & 5.9e-07 \\ \hline
Share1b      & 117  & 253   & 102  & 238   & 22  &   -7.6589e+04  & 6.5e-08 & 25 & -7.6589e+04 & 1.5e-06 \\ \hline
Share2b      & 96   & 162   & 87   & 153   & 13  &   -4.1573e+02  & 4.9e-11 & 15 & -4.1573e+02 & 7.9e-10 \\ \hline
Ship04l      & 402  & 2166  & 292  & 1905  & 17  &   1.7933e+06   & 5.2e-11 & 18 &  1.7933e+06 & 2.9e-11 \\ \hline
Ship04s      & 402  & 1506  & 216  & 1281  & 17  &   1.7987e+06   & 2.2e-11 & 20 &  1.7987e+06 & 4.5e-09 \\ \hline
Ship08l**    & 778  & 4363  & 470  & 3121  & 18  &   1.9090e+06   & 1.6e-07 & 20 &  1.9091e+06 & 1.0e-10 \\ \hline
Ship08s+     & 778  & 2467  & 274  & 1600  & 17  &   1.9201e+06   & 3.7e-08 & 19 &  1.9201e+06 & 4.5e-12 \\ \hline
Ship12l*     & 1151 & 5533  & 610  & 4171  & 19  &   1.4702e+06   & 4.7e-13 & 20 &  1.4702e+06 & 1.0e-08 \\ \hline
Ship12s+     & 1151 & 2869  & 340  & 1943  & 17  &   1.4892e+06   & 1.0e-10 & 19 &  1.4892e+06 & 2.1e-13 \\ \hline
Stocfor1*    & 117  & 165   & 34   & 82    & 14  &   -4.1132e+04  & 2.8e-10 & 15 & -4.1132e+04 & 1.1e-10 \\ \hline
Stocfor2     & 2157 & 3045  & 766  & 1654  & 22  &   -3.9024e+04  & 2.1e-09 & 22 & -3.9024e+04 & 1.6e-09 \\ \hline
Stocfor3     & 16675& 23541 & 5974 & 12840 & 34  &  -3.9976e+04   & 4.7e-08 & 38 & -3.9976e+04 & 6.4e-08 \\ \hline
Truss        & 1000 & 8806  & 1000 & 8806  & 22  &  4.5882e+05    & 1.7e-07 & 36 &  4.5882e+05 & 9.5e-06 \\ \hline
\caption{Numerical results for test problems in Netlib}
\end{longtable}
\normalsize

Several problems have degenerate solutions which make them difficult to solve or need significantly more 
iterations. We choose to use the option described in Section 4.6 to solve these problems. 
For problems marked with '+', this option is called only for 
Mehrotra's method. For problems marked with '*', both algorithms need to call this option for better results. 
For the problem with '**', in addition to call this option, the default value of 
$10^{-6}$ has to be changed to $10^{-4}$ for Mehrotra's method.
We need to keep in mind that although using the option described in 
Section 4.6 reduces the iteration count significantly,
these iterations are significantly more expensive. 
Therefore, simply comparing 
iteration counts for problem(s) marked with '+' will lead to
a conclusion in favor of Mehrotra's method (which is what we will do in the following discussions).

Since the major cost in each iteration for both algorithms are solving linear systems of 
equations,
which are identical in these two algorithms, we conclude that iteration
numbers is a good measure of efficiency. In view of Table 1, it is 
clear that Algorithm \ref{mainAlgo3} uses less iterations
than Mehrotra's algorithm to find the optimal solutions for majority tested problems. Among $51$ 
tested problems, Mehrotra's method uses fewer iterations ($7$ iterations in total) than arc-search 
method for only $6$ problems ({\tt brandy}, {\tt osa\_07}, {\tt sc205}, {\tt sc50a},
{\tt scagr25}, {\tt scfxm3}\footnote{For this problem, Mehrotra's method needs to use the option
described in Section 4.6 but arc-search method does not need to. As a result, Mehrotra's method
uses noticeably more CPU time then arc-search method.}), while arc-search method uses fewer 
iterations ($126$ iterations in total)
than Mehrotra's method for $40$ problems. For the rest $5$ problems, both methods use 
the same number of iterations. Arc-search method is numerically more stable than 
Mehrotra's method because for problems {\tt scfxm1}, {\tt scfxm3}, {\tt ship08s}, {\tt ship12s}, 
arc-search method does not need to use the option described in Section 4.6 but Mehrotra's method need to
use the option to solve the problems. For problem {\tt ship08l}, Mehrotra's method need to adjust parameter 
in the option to find the optimizer but arc-search method does not need to adjust the parameter.

\section{Conclusions}

This paper proposes an arc-search interior-point path-following algorithm that 
searches optimizers along the ellipse that approximate infeasible central path. The proposed algorithm 
is different from Mehrotra's method
only in search path. Both arc-search method and Mehrotra's method
are implemented in Matlab so that the two methods use exactly same
initial point, the same pre-process, the same parameters, and the same 
stopping criteria. By doing this, we can compare both algorithms in 
a fair and controlled way.
Numerical test is conducted for Netlib problems for both methods. 
The results show that the proposed
arc-search method is more efficient and reliable than the well-known Mehrotra's method.

\section{Acknowledgments}
  
The author would like to thank Mr. Mike Case, the Director of the Division of Engineering in the Office of 
Research at US NRC, and Dr. Chris Hoxie, in the Office of Research at US NRC, for their providing computational 
environment for this research.


\end{document}